\documentclass[11pt]{article}
\usepackage{latexsym,amsmath,amsthm,verbatim,ifthen,amssymb}
\usepackage[latin1]{inputenc}
\usepackage[all]{xy}
\usepackage{graphicx}
\usepackage{epsfig}

\newdir{ >}{{}*!/-12pt/@{>}}
\newdir{ (}{!/-3pt/@_{(} } 

\newdir{ )}{!/-3pt/@^{(} } 

\newtheorem{theorem}{Theorem}[section]

\newtheorem{corollary}[theorem]{Corollary}
 \newtheorem{lemma}[theorem]{Lemma}
 \newtheorem{proposition}[theorem]{Proposition}
 \theoremstyle{definition}
 \newtheorem{definition}[theorem]{Definition}
 \theoremstyle{remark}
 \newtheorem{remark}[theorem]{Remark}
 
 \numberwithin{equation}{subsection}

\newcommand{\cat}{{\sf {cat}}} 
\newcommand{\secat}{{\sf {secat}}} 
\newcommand{\tc}{{\sf {TC}}} 

\newcommand{\sn}{{\sf {sec}}} 
\newcommand{\D}{{\sf {D}}}

\begin{document}

\title{Generalized topological complexity and its monoidal version}

\author{J.M. Garc\'{\i}a-Calcines\footnote{Universidad de La Laguna,
Facultad de Ciencias, Departamento de Matem\'aticas,
Estad\'{\i}stica e I.O., 38271 La Laguna, Spain. E-mail:
\texttt{jmgarcal@ull.edu.es}} }

\maketitle

\begin{abstract}
In the context of the Lusternik-Schnirelmann category, researcher T. Srinivasan demonstrated that when the space under consideration is an absolute neighborhood retract, its category can be realized through arbitrary subsets, not necessarily open ones. The primary aim of this survey is to illustrate how this result has been extended to the case of topological complexity and its monoidal version, along with some of its most significant implications.
\end{abstract}

\vspace{0.5cm}
\vspace{0.2cm}

\section*{Introduction}
In a series of two articles, \cite{Sr,Sr2}, the mathematician T. Srinivasan, a student of A. Dranishnikov, introduced the concept of generalized Lusternik-Schnirelmann category. The idea behind this concept is quite simple; it involves considering coverings by arbitrary contractible subsets in the space, instead of taking open (or closed) subsets with this property. The generalized LS category serves as a lower bound for the classical LS category, and the two may not necessarily coincide. However, what is truly remarkable in T. Srinivasan's work is that she managed to demonstrate that if the space under consideration is an absolute neighborhood retract (ANR space or just ANR, for short), then the notion of classical LS category coincides with its generalized version. This opens an important avenue in the development of this homotopical numerical invariant, as it could greatly simplify its calculation in many cases, a matter that has yet to be fully explored.

After this magnificent result, it is natural to wonder whether it can be extended to a broader context, such as that of topological complexity in the sense of M. Farber \cite{F}. The answer is affirmative, as we are going to display in this survey. The concept of generalized topological complexity emerges, which does not require open coverings but rather coverings by arbitrary subsets. We will observe that when the space in question has the homotopy type of an ANR space (or of a CW-complex), then the topological complexity of that space coincides with its generalized version. The broad scope of this class of spaces, combined with the lack of specific requirements imposed on the coverings, significantly simplifies the search for explicit motion planning algorithms, as they no longer need to be tame (see \cite[Def. 4.4]{F2}).

There are two approaches to generalized topological complexity through which the analogous result to Srinivasan's for the LS category is proven. Firstly, we have the approach given by P. Pa\v{v}esi\'{c} \cite{Pav}, focusing on the topological complexity of a continuous map. It was within this framework that the initial results establishing the equality between topological complexity and its generalized counterpart were revealed. The only restriction is that Pa\v{v}esi\'{c} considers compactness among ANR spaces, possibly influenced by T. Srinivasan's early work \cite{Sr}, where compact ANR spaces were also considered.
The second approach, subsequent to Pa\v{v}esi\'{c}'s, was provided by J. Calcines \cite{GC} through the sectional category or Schwarz genus. In this context, it is demonstrated that the generalized and standard sectional categories coincide for fibrations between ANR spaces. The broad range of homotopical numerical invariants covered by the sectional category results in the recovery of Srinivasan's outcome but also gives rise to the corresponding results for topological complexity, sequential topological complexity, and homotopic distance, among others.

An interesting variant of topological complexity is the so-called monoidal topological complexity, introduced by N. Iwase and M. Sakai \cite{I-S}. Here, open coverings are utilized, with each covering integrating a local motion planner, and mandating that the output path must always be the constant path when the initial and final points coincide. In this case as well, it has been observed in \cite{GC} that monoidal topological complexity can be characterized using arbitrary coverings, provided ANR spaces are taken into account. Analogous to the situation with topological complexity, this result can be situated within the broader framework of the relative category in the sense of Doeraene and El Haouari.

\section{Preliminaries: Srinivasan's tools and generalized LS category}

In this introductory section, we delve into the tools employed by T. Srinivasan, as they will be instrumental in the subsequent section, where we discuss an extension related to topological complexity. Additionally, we examine how T. Srinivasan successfully reached her objectives with the LS category using these straightforward tools.

\subsection{Absolute neighborhood retracts}

In the early 1930s of the 20th century, Karol Borsuk introduced the concept of a retract \cite{Bor} and also laid the foundation for absolute retract and absolute neighborhood retract \cite{Bor2}. The study of these spaces became so extensive that, by 1967, two notable books had already been dedicated to them: \emph{Theory of Retracts}, by K. Borsuk \cite{Bor3} and \emph{Theory of Retracts} by S.T. Hu \cite{Hu}. In this study, our main focus centers on ANR spaces. To recap, if we have a metrizable space $X,$ it is considered an ANR, or an ANR space, if, for any embedding $X\subset Y$ as a closed subspace in a metrizable space $Y$, $X$ is a neighborhood retract of $Y$; this means that there exists an open neighborhood $U$ of $X$ in $Y$ that can be retracted onto $Y$. An alternative, and perhaps more comprehensive, widely used definition is as follows: A metrizable space $X$ is an ANR if, and only if, it is an absolute neighborhood extensor (ANE). That is, for each closed subset $A$ of a metrizable space $W$ and every continuous map $f : A\rightarrow X,$ there exists an open neighborhood $U$ of $A$ in $W$ along with a continuous map $F : U\rightarrow X,$ extending $f$.

The class of ANR spaces exhibits numerous intriguing properties. Notably, a significant observation is that being an ANR space is a local characteristic for metrizable spaces. In fact, a metrizable topological space that can be covered by open subsets that are ANR spaces is itself an ANR space (see \cite[Chapt. III, Th. 8.1]{Hu} or \cite[Th. 3.3]{Han}). Moreover, it is easy to show that any open subset or retract of an ANR is also an ANR, and that the finite product of ANR spaces remains an ANR space \cite[Chapt. II, Prop. 4.4]{Hu}.

The class of ANR spaces extends far beyond initial assumptions. For instance, it encompasses every finite-dimensional metrizable locally Euclidean topological space, as demonstrated in \cite[Chapt. III, Cor. 8.3]{Hu}. Notably, this includes all topological manifolds. Another illustration of an ANR space is provided by any locally finite CW-complex.
In addition, we have a result that is highly valuable when considering homotopy invariants, such as topological complexity or Lusternik-Schnirelmann category. Specifically (see \cite{Mil}, for instance):

\begin{theorem}
A space $X$ has the homotopy type of a CW-complex if, and only if, it has the homotopy type of an ANR.
\end{theorem}

\subsection{Modified Walsh Lemma}

In this subsection, our goal is to introduce essential tools that contribute to the development of T. Srinivasan's theory. One of these crucial tools is the Modified Walsh Lemma. To achieve this, we will begin by exploring various preceding results, some of which are widely recognized. For detailed proofs and more comprehensive information, readers can refer to \cite{Hu}, \cite{Bor3}, and \cite{Sr2}.

To introduce our first result in this section, we require the following definitions. Let $X$ be a space and consider $\mathcal{U}=\{U_{\lambda }\}_{\lambda \in \Lambda }$ a covering of a space $Y$:

\begin{enumerate}
\item[(i)] We define two maps $f,g:X\rightarrow Y$ as being \emph{$\mathcal{U}$-near} if, for any $x\in X$, there exists a $\lambda \in \Lambda $ such that both $f(x)$ and $g(x)$ belong to $U_{\lambda }.$

\item[(ii)] A map $H:X\times I\rightarrow Y$ is considered an \emph{$\mathcal{U}$-homotopy} when, for any $x\in X$, there exists a $\lambda \in \Lambda $ such that $H(\{x\}\times I)$ is contained within $U_{\lambda }.$
\end{enumerate}

\begin{proposition}\cite[Chapt. IV, Th. 1.1]{Hu}
Let $Y$ be an ANR and $\mathcal{U}=\{U_{\lambda }\}_{\lambda \in \Lambda }$ an open cover of $Y$. Then, there exists an open cover $\mathcal{V}=\{V_{\mu }\}_{\mu \in M}$ of $Y$ (which is an refinement of $\mathcal{U}$) such that for any two $\mathcal{V}$-near maps $f,g:X\rightarrow Y$ defined on an arbitrary space $X$, there exists a $\mathcal{U}$-homotopy connecting them.
\end{proposition}

Looking carefully at the proof of the previous result, we can find a particular result that will be very valuable.

\begin{proposition}\label{epsilon}
If $Y$ is a bounded metric ANR, then there exists a (not necessarily continuous) function $\varepsilon :Y\rightarrow (0,+\infty )$ such that the cover formed by the collection  $\mathcal{V}=\{B(y,\varepsilon (y)):y\in Y\}$ of open balls in $Y$ is such that any pair of maps $f,g:X\rightarrow Y$, which are $\mathcal{V}$-near, are homotopic.
\end{proposition}

\begin{remark}\label{compact}
Specifically, if $Y$ is also compact, it is possible to find a positive real number $\varepsilon > 0$ such that for any pair of maps $f, g: X \rightarrow Y$ that are $\varepsilon$-close (i.e., $d(f(x), g(x)) < \varepsilon$ for all $x \in X$), they are homotopic (compare with \cite[Th. 2.4]{Sr} or \cite{Bor3}).
\end{remark}

\begin{remark}
Note that for any metrizable space $Y$, one can choose a bounded metric that induces the space's topology. When $Y$ is compact, all metrics are necessarily bounded.
\end{remark}

\medskip
Among the tools employed by T. Srinivasan in her work, the most significant is arguably the \emph{Modified Walsh Lemma}, which is a variation of a lemma found in \cite{W}.
You can find the proof presented in T. Srinivasan's work \cite{Sr2}.

\begin{theorem}\cite[Th. 2.3]{Sr2}\label{walsh}
Let $X$ be a metric space, $A\subset X$, $K$ a bounded metric ANR and $f:A\rightarrow K$ a continuous map. Then, for any function $\varepsilon :K\rightarrow (0,+\infty )$ (not necessarily continuous) there exists an open neighborhood $U\supset A$ in $X$ and a continuous map $g:U\rightarrow K$ such that:
\begin{enumerate}
\item For every $u\in U$ there exists $a_u\in A$ such that $d(g(u),f(a_u))<\varepsilon (f(a_u))$.

\item $g_{|A}\simeq f.$
\end{enumerate}
\end{theorem}

\begin{remark}\label{coment}
As can be verified in the details of the theorem proof, one can even choose $a_u$ in such a way that the distance $d(u,a_u)$ can be made as small as required.
\end{remark}

\begin{remark}
An earlier version of the theorem can be found in \cite[Theorem 2.3]{Sr}, which was specifically tailored for the compact case. Under quite similar conditions as stated in the theorem, for any $\varepsilon > 0$, it is possible to discover an open subset $U$ that contains $A$ in the space $X$ and a map $g: U \rightarrow K$ such that the the image $g(U)$ lies within an $\varepsilon$-neighborhood of $f(A)$. Additionally, the restriction of $g$ to $A$ is homotopy equivalent to $f$.
\end{remark}

We would like to conclude this section by illustrating how T. Srinivasan used all these tools for the Lusternik-Schnirelmann category and its generalized version.

\begin{definition}
Given a map $f:X\rightarrow Y$, the generalized Lusternik-Schnirelmann category of $f$ (or just LS-category of $f$), denoted as $\cat _g(f)$, is the smallest non-negative integer $n$ for which $X$ can be covered by $n+1$ subsets
$X=A_0\cup A_1\cup \cdots A_n$ where each restriction $f_{|A_i}:A_i\rightarrow Y$ is nullhomotopic. If such integer does not exist, then we set $\cat _g(f)=\infty $.
\end{definition}

 Taking open covers in the definition above we obtain the standard notion of $\cat (f).$  The generalized LS-category of a space $X$ is then defined as
$$\cat _g(X):=\cat _g(id_X)$$

If $f:X\rightarrow Y$ is a map where $Y$ is path-connected, and we fix a base point $y_0\in Y$ it is a well-known fact that $f$ is nullhomotopic if and only if there exists a factorization of $f$ of the form
$$\xymatrix{
{X} \ar[rr]^f \ar[dr]_F & & {Y} \\
 & {PY} \ar[ur]_p &  }$$
Here, $PY=\{\alpha \in Y^I:\alpha (0)=y_0\}$ as a subspace of $Y^I$ (with the compact-open topology) and $p:PY\rightarrow Y$ is the path fibration, given as $p(\alpha ):=\alpha (1).$

We also observe that, if $Y$ is a metric space, then so is $Y^I$ (and $PY$) considering the metric $d^*$ of the supremum, which induces the compact-open topology. More is true, if $Y$ is an ANR space, then $Y^I$ and $PY$ are ANR spaces (see \cite[Chapt. VI, Th. 2.4]{Hu} and \cite[Chapt. VI, Th. 3.1]{Hu}).

\begin{lemma}
Let $X$ be a metric space, $A\subset X$, $K$ a connected ANR and $f:X\rightarrow K$ a map such that the restriction $f_{|A}:A\rightarrow K$ is nullhomotopic. Then there exists an open neighbourhood $U\supset A$ in $X$ such that $f_{|U}:U\rightarrow K$ is nullhomotopic.
\end{lemma}

\begin{proof}
If we consider a bounded metric on $K$, this ensures that the metric of the supremum, denoted as $d^*$ in $PK$, remains bounded as well. For simplicity, let us denote the metrics in both $X$ and $K$ as $d$.
Given that $K$ is a bounded metric ANR, there exists a function $\varepsilon : K \rightarrow (0,+\infty)$ that satisfies the conditions outlined in Proposition \ref{epsilon}.
Leveraging the continuity of $f$, for each $a \in A$, we can select $\delta_a > 0$ such that if $d(x,a) < \delta_a$, it implies $d(f(x),f(a)) < \varepsilon(f(a))$.
Now, considering that $f_{|A}$ is nullhomotopic, we define a map $F: A \rightarrow PK$ such that $p \circ F = f_{|A}$. We can apply Theorem \ref{walsh} with $\overline{\varepsilon} : PK \rightarrow (0,+\infty)$ defined as $\overline{\varepsilon}(\alpha) := \varepsilon(\alpha(1))$.
This enables us to find an open neighborhood $U$ that contains $A$ in the space $X$, and a map $G: U \rightarrow PK$ satisfying that for any $u \in U$, there exists an associated $a_u \in A$ such that $d^*(G(u), F(a_u)) < \overline{\varepsilon}(F(a_u))=\varepsilon (f(a_u))$. Furthermore, we can choose $a_u$ in such a way that $d(u, a_u) < \delta_{a_u}.$

Now, let us define a map $g: U \rightarrow K$ as $g(u) := G(u)(1)$. Then, we can establish the following inequality:
$$d(g(u),f(a_u))=d(G(u)(1),F(a_u)(1))\leq d^*(G(u),F(a_u))<\varepsilon (f(a_u))$$
Furthermore, since $d(u, a_u) < \delta_{a_u}$, we also have that $d(f(u), f(a_u)) < \varepsilon(f(a_u))$. In other words, both $f(u)$ and $g(u)$ fall within the open ball $B(f(a_u), \varepsilon(f(a_u))$. According to Proposition \ref{epsilon}, this implies that $f_{|U} \simeq g$. As $g$ is obviously nullhomotopic, it follows that $f_{|U}$ is also nullhomotopic.
\end{proof}

As an immediate consequence of the lemma we obtain this beautiful result by T. Srinivasan:

\begin{theorem}
Consider a map $f: X \rightarrow Y$, where $X$ is a metrizable space and $Y$ is a connected ANR space. Then, it holds that $\cat_g(f) = \cat(f).$ In particular, when $X$ is a connected ANR space, $\cat_g(X) = \cat(X).$
\end{theorem}

\section{Generalized topological complexity}

As we have just discussed earlier, the generalized Lusternik-Schnirelmann category for a space $X$, $\cat_g(X)$, coincides with $\cat(X)$ when $X$ is an ANR. In this section, we will examine how this result can be extended to topological complexity. The concept of generalized topological complexity is defined logically by considering coverings with arbitrary, not necessarily open, subsets:

\begin{definition}\label{Gtc} The \emph{generalized topological complexity} of a space $X$, denoted by $\tc _g(X)$, is the least nonnegative integer $n$ (or infinity if such integer does not exist) such that $X\times X$
admits a cover by $n+1$ subsets $A_0,...,A_n$ on each of which there exists a continuous local section of the bi-evaluation path fibration
$\pi _X:X^I\rightarrow X\times X$, $\pi _X(\alpha )=(\alpha (0),\alpha (1)).$
\end{definition}

It is quite evident that $\tc _g(X) \leq \tc (X)$. In this section, we will explore conditions on a space $X$ that are sufficient to establish $\tc_g(X) = \tc(X).$
There have been two different approaches to address this result. The first mathematician to achieve positive results in this direction was P. Pa\v{v}esi\'{c}.

\subsection{P. Pa\v{v}esi\'{c}'s approach through topological complexity of a map}

In \cite{Pav}, P. Pa\v{v}esi\'{c} developed the notion of topological complexity of a map $f:X\rightarrow Y$, an intriguing invariant that generalizes both topological complexity and the Lusternik-Schnirelmann category. This concept was first introduced in \cite{Pav2}, where it served as a measure to quantify the manipulation complexity of a robotic device. In this context, $X$ and $Y$ represented the configuration space and working space, respectively, of a mechanical system, such as a robot arm. The map $f:X\rightarrow Y$ was interpreted as the forward kinematic map of the system. Subsequently, this theory was expanded upon in \cite{Pav}. It was in this context that the first results establishing the equality $\tc_g(X) = \tc(X)$ came to light. The only constraint is that Pa\v{v}esi\'{c} considers compactness among ANR spaces, possibly influenced by T. Srinivasan's initial work \cite{Sr}, where she also considered compact ANR spaces. Due to its significance and its connection with the topological complexity of a map, we dedicate this subsection to these early stages of generalized topological complexity.

Consider a continuous surjection $f: X \rightarrow Y$. The sectional number $\sn (f)$, is defined as the smallest integer $n$ for which there exists a sequence of nested open subsets:
$$\emptyset \subset U_0\subset U_1\subset \cdots \subset U_{n+1}=Y$$
\noindent such that each difference $U_i \setminus U_{i-1},$ $i\in \{1, \ldots, n+1\}$ permits a local section of $f$. If there is no such integer $n$, then we set $\sn (f) = \infty$.

It is relevant to point out that this definition deviates from the standard notion of sectional number, which is typically defined as one less than the minimum number of elements in an open cover of $Y$ where each element allows for a local section of $f$. We denote this alternative quantity as $\sn_{op}(f)$. Clearly, $\sn(f) \leq \sn_{op}(f)$. However, it is not difficult to check that when $f$ is a fibration and $Y$ is an ANR space, we have $\sn(f) = \sn_{op}(f)$.

It is also interesting to draw a parallel between $\sn_{op}(f)$ and $\secat(f)$, the sectional category of $f,$ also known as the Schwarz genus of $f$. First, we recall the definition of sectional category of a map:

\begin{definition} Let $f:X\rightarrow Y$ be a map. Then, the sectional category of $f,$ denoted by $\secat(f)$, is the smallest non-negative integer $n$ (or infinity), for which $Y$ can be covered by $n+1$ open subsets, and on each of these open subsets, there exists a local homotopy section of $f$.
\end{definition}

The term 'sectional category' was initially introduced by Schwarz in \cite{Sch}, originally referred to as 'genus,' and later renamed by James in \cite{J}. The sectional category is a variant of Lusternik-Schnirelmann category and it serves as a generalization. Specifically, when $X$ is a contractible space, $\secat(f)$ coincides with $\cat(Y)$.
Notably, $\secat(f)=\sn _{op}(f)$ if $f$ is a fibration. However, it is essential to be aware that in general, $\sn(f)$ can be significantly larger than $\secat(f)$ (see \cite[Section 5]{Pav}) for examples).

Given $f:X\rightarrow Y$ a continuous surjection between path-connected spaces one can consider the map
$$\pi _f:X^I\rightarrow X\times Y,\hspace{10pt}\alpha \mapsto (\alpha (0),f(\alpha (1)))$$

\begin{definition}
Let $f:X\rightarrow Y$ be a continuous surjection between path-connected spaces. The topological complexity of $f$, denoted by $\tc (f)$, is defined as
$\tc (f):=\sn (\pi _f:X^I\rightarrow X\times Y)$.
\end{definition}

Clearly, $\tc (X)=\tc (id_X)$, for any ANR space $X.$ Moreover, by \cite[Cor. 4.8]{Pav} we also have the identity $\cat (X)=\tc (p:PX\rightarrow X).$
Considering that, as clarified in \cite[Lem. 4.1]{Pav}, $f$ is a fibration if and only if $\pi_f$ is a fibration, we can deduce that
$$\tc (f)=\secat (\pi _f:X^I\rightarrow X\times Y)$$
\noindent whenever $f:X\rightarrow Y$ is a fibration between ANR spaces.

The subsequent outcome concerning a fibration $f$ offers a highly valuable characterization of the circumstances in which a subset $A\subset X\times Y$ allows for a local section of $\pi _f$. We will denote the canonical projections as $pr_1:X\times Y\rightarrow X$ and $pr_2:X\times Y\rightarrow Y.$

\begin{lemma}\label{caracter} \cite[Prop. 4.3]{Pav}
Let $f:X\rightarrow Y$ be a fibration and $A\subset X\times Y$ a subset. Then the following statements are equivalent:
\begin{enumerate}
\item[(i)] There is a local section $s:A\rightarrow X^I$ of $\pi _f$.

\item[(ii)] The maps $(f\circ pr_1)_{|A}, (pr_2)_{|A}:A\rightarrow Y$ are homotopic.
\end{enumerate}
\end{lemma}

\begin{proposition}\label{import}
Let $f,g:X\rightarrow Y$ maps between compact metric ANR spaces and $A\subset X$ a subset. If $f_{|A}\simeq g_{|A},$ then there exists an open subset $U$ containing $A$ such that $f_{|U}\simeq g_{|U}.$
\end{proposition}

\begin{proof}
For simplicity, let us denote the metrics in both $X$ and $Y$ as $d$, and the metric of the supremum in $Y^I$ as $d^*$. In accordance with Remark \ref{compact}, let us choose $\varepsilon > 0$ such that any pair of maps mapping to $Y$ that are $\varepsilon$-close are homotopic. Since $X$ is compact, both $f$ and $g$ are uniformly continuous, implying the existence of $\delta > 0$ such that $d(x, x') < \delta$ leads to $d(f(x), f(x')) < \frac{\varepsilon}{2}$ and $d(g(x), g(x')) < \frac{\varepsilon}{2}$.

Now, consider $F: A \times I \rightarrow Y$, a homotopy between $f_{|A}$ and $g_{|A}$, and $\hat{F}: A \rightarrow Y^I$, its corresponding adjoint map. By applying Theorem \ref{walsh} and Remark \ref{coment}, one can select an open neighborhood $U\supset A$ and a map $G: U \rightarrow Y^I$ in such a way that for all $u \in U$, there exists $a_u \in A$ satisfying $d(u, a_u) < \delta$ and $d^*(G(u), \hat{F}(a_u)) < \frac{\varepsilon}{2}$ (i.e., $d(G(u)(t), \hat{F}(a_u)(t)) < \frac{\varepsilon}{2}$ for all $t \in I).$
Define $G_0$ and $G_1: U \rightarrow Y$ as $G_0(u) := G(u)(0)$ and $G_1(u) := G(u)(1)$. By observing that $\hat{F}(a_u)(0) = F(a_u, 0) = f(a_u)$ and applying the triangle inequality, we get:
$$d(G_0(u),f(u))\leq d(G_0(u),\hat{F}(a_u)(0))+d(f(a_u),f(u))<\frac{\varepsilon }{2}+\frac{\varepsilon }{2}=\varepsilon .$$
This implies that $G_0$ and $f_{|U}$ are $\varepsilon$-close, and hence $G_0 \simeq f_{|U}$. Similarly, $G_1 \simeq g_{|U}$. Since $G_0 \simeq G_1$, we conclude that $f_{|U} \simeq g_{|U}$.
\end{proof}

Now we are ready to state and demonstrate the main result in this subsection.

\begin{theorem}\label{pavesic}
Let $f:X\rightarrow Y$ be a fibration between compact ANR spaces. Then $\tc (f)$ is the smallest nonnegative integer $n$ for which $X\times Y$ admits a cover by $n+1$ subsets
on each of which there exists a local section of $\pi _f$.
\end{theorem}

\begin{proof}
If $s: A \rightarrow X^I$ is a local section of $\pi _f$, according to Lemma \ref{caracter}, it follows that the maps $(f\circ pr_1){|A}$ and $(pr_2){|A}$ are homotopic. Employing Proposition \ref{import}, we can establish an open neighborhood $U$ containing $A$ within $X\times Y$, such that the maps $(f\circ pr_1){|U}$ and $(pr_2)_{|U}$ are homotopic. By applying Lemma \ref{caracter} again, this implies the existence of a local section $s': U \rightarrow X^I$ of $\pi _f$. By extending this argument to covers, we can deduce the result.
\end{proof}

The following result is then immediate:

\begin{corollary}
If $X$ is a compact ANR, then $\tc (X)=\tc _g(X).$ If, in addition, $X$ is connected, then $\cat (X)=\cat _g(X).$
\end{corollary}

\subsection{J. Calcines' approach through sectional category}\label{calcines}

We will explore the extension of Srinivasan's results from a different approach, considering a more general invariant, the sectional category. This extension was introduced by J. Calcines in his work \cite{GC}, achieving the desired results without the restriction of compactness. Our primary sources of reference for this subsection will be \cite{GC} and \cite{Sr2}. We will initiate our investigation by exploring the concept of generalized sectional category, which is defined in a manner analogous to the generalized LS category or topological complexity.

\begin{definition}\cite[Def. 2.1]{GC}
Let $f:X\rightarrow Y$ be a map. The \emph{generalized sectional category of $f$}, denoted as $\secat_g(f),$ is the least nonnegative integer $n$ such that $B$ admits
a cover by $n+1$ subsets $A_0,...,A_n$ on each of which there exists a local homotopy section of $f$.
If such an integer does not exist, then we set $\secat_g(p)=\infty .$
\end{definition}

As in the standard case, when $f: X \rightarrow Y$ is a fibration, we can equivalently consider strict local sections due to the homotopy lifting property. Extending the concept of generalized sectional category to encompass general maps not only offers a more versatile framework but also establishes connections between topological invariants across a broader range of scenarios.

A fundamental property that relates the generalized versions of the sectional category for two maps is presented in the following result.
If there exists a homotopy commutative diagram in the form of
$$\xymatrix{
{X} \ar[rr]^{\alpha } \ar[dr]_f & & {X'} \ar[dl]^{f'} \\  & Y &}$$
\noindent then, we have the inequality $\secat_g(f') \leq \secat_g(f)$.
In fact, if we consider any subset $A$ of $Y$ along with a local homotopy section $s: A \rightarrow X$ of $f$, it becomes clear that
$s':=\alpha \circ s: A \rightarrow X'$ is a local homotopy section of $f'$. By reasoning with covers, we can derive the inequality.

\medskip
As we will state below, the generalized sectional category is a homotopy invariant. We leave the proof of this result to the reader.

\begin{proposition}\label{invariance}
Let us consider the following homotopy commutative diagram, in which $\alpha$ and $\beta$ are homotopy equivalences:
$$\xymatrix{
{X} \ar[r]^{\alpha }_{\simeq } \ar[d]_f & {X'} \ar[d]^{f'} \\
{Y} \ar[r]^{\simeq }_{\beta } & {Y'} }$$ In this situation, we have $\secat_g(f) = \secat_g(f').$
\end{proposition}

The very first example of generalized sectional category is the generalized Lusternik-Schnirelmann category. Indeed, for a path-connected $X$
with a base point $x_0\in X$, any contractible set in $X$, $A$, can be deformed to the base point so that the inclusion $A\hookrightarrow X$ is, actually, homotopic to the constant path at $x_0$. Consequently, like in the standard case, $\cat _g(X)$ can be equivalently given as
$\cat _g(X)=\secat _g(*\stackrel{x_0}{\longrightarrow }X)$. Here $x_0:*\rightarrow X$ represents the constant map at $x_0.$ Notably, when we consider the path fibration
$p:PX\rightarrow X$, we can use Proposition
\ref{invariance} to conclude that, in fact,
$\cat _g(X)=\secat _g(p:PX\rightarrow X)$.
Observe that the fibration $p$ approximates the constant map $x_0:*\rightarrow X.$

The fundamental example of generalized sectional category is the \emph{generalized topological complexity} of a space. Indeed, directly from the definition, it can be verified that $\tc _g(X)=\secat _g( \pi _X:X^I\rightarrow X\times X).$
Once again, since $\pi_X$ is the fibration that serves as an approximation of the diagonal map $\Delta_X: X \rightarrow X \times X$, as per Proposition \ref{invariance}, we can establish the equality
$\tc _g(X)=\secat _g(\Delta _X:X\rightarrow X\times X)$.

\medskip
As a result of Definition \ref{Gtc}, we derive a corollary from Proposition \ref{invariance}:
\begin{corollary}\label{hom-inv}
$\tc _g(X)$ depends only on the homotopy type of the space $X.$
\end{corollary}

\begin{proof}
If $X$ and $X'$ are homotopy equivalent spaces, then the diagonal maps $\Delta _X$ and $\Delta _{X'}$ fit in a homotopy commutative square where they are connected through homotopy equivalences. Therefore we can apply Proposition \ref{invariance}.
\end{proof}

Now, consider a map $f:X\rightarrow Y$, where $Y$ is a path-connected based space. The following diagram is commutative up to homotopy:
$$\xymatrix{ {*} \ar[rr] \ar[dr]_{y_0} & & {X} \ar[dl]^{f} \\
& Y &}$$ \noindent Hence, we can utilize the argument presented before Proposition \ref{invariance} to conclude that $\secat _g(f)\leq \cat _g(Y).$ Furthermore, if $X$ is contractible, the map $*\rightarrow X$ becomes a homotopy equivalence, implying $\secat _g(f)=\cat _g(Y)$ by Proposition \ref{invariance}. As a special case, we have $\tc_g(X) \leq \cat _g(X\times X)$ for any path-connected space $X.$ Finally, one can follow the same proof as as outlined in \cite[Theorem 5]{F} to establish that:
$$\cat_ g(X)\leq \tc _g(X)\leq \cat _g(X\times X).$$

\medskip
Now we prove that generalized sectional category agrees with sectional category of a fibration provided we are dealing with absolute neighborhood retracts. In our particular context, we must make slight modifications to the original statement of the Modified Walsh Lemma, as presented in J. Calcines' work \cite[Lem. 2.6]{GC}. To prove this result, one simply needs to emulate the approach outlined in \cite[Theorem 2.3]{Sr2}, carefully adapting it to ensure that the distances $d(g(u),f(a_u))$ and $d(u,a_u)$ are both less than $\varepsilon (f(a_u)).$

\begin{theorem}\cite[Lem. 2.6]{GC}, \cite[Th. 2.3]{Sr2}\label{walsh2}
Let $X$ be a metric space, $A\subset X$, $K$ a bounded metric ANR space and $f:A\rightarrow K$ a continuous map. Then, for any function $\varepsilon :K\rightarrow (0,+\infty )$ (not necessarily continuous) there exists an open neighborhood $U\supset A$ in $X$ and a continuous map $g:U\rightarrow K$ such that:
\begin{enumerate}
\item For every $u\in U$ there exists $a_u\in A$ such that
$$\max \{d(g(u),f(a_u)),d(u,a_u))\}<\varepsilon (f(a_u))$$

\item $g_{|A}\simeq f.$
\end{enumerate}
\end{theorem}

Now we are ready for the statement and proof of the main result in this section.

\begin{theorem}\cite[Th. 2.7]{GC}\label{chulo}
Let $p:E\rightarrow B$ be a fibration where $E$ and $B$ are ANR spaces. Then $\secat(p)=\secat _g(p).$
\end{theorem}

\begin{proof}
Since $B$ is an ANR, Proposition \ref{epsilon} implies the existence of a function $\varepsilon :B\rightarrow (0,+\infty )$ such that the cover $\mathcal{W}=\{B(b,\varepsilon (b)):b\in B\}$ of $B$ has the property that any pair of maps $f,g:Z\rightarrow B$ that are $\mathcal{W}$-close are homotopic.

We utilize continuity of the map $p:E\rightarrow B$ to select, for each $e\in E$, a $\delta (e)>0$ such that $p(B(e,\delta (e)))\subset B(p(e),\varepsilon (p(e))).$ We then define the function $\bar{\varepsilon }:E\rightarrow (0,+\infty )$ as $\bar{\varepsilon }(e)=\min \{\delta (e),\varepsilon (p(e))\}.$

Now let us consider the scenario where we have a subset $A\subset B$ along with a strict local section $s:A\rightarrow E$ of $p.$
According to Theorem \ref{walsh2}, we can find an open neighborhood $U\supset A$ in $B$ and a map $s':U\rightarrow E$ satisfying that for all $u\in U$, there exists $a_u\in A$ such that $\max \{d(s'(u),s(a_u)),d(u,a_u))\}<\bar{\varepsilon }(s(a_u))$.
Now, if $inc_U:U\hookrightarrow B$ denotes the canonical inclusion, then it is not very difficult to check that $p\circ s'$ and $inc_U:U\rightarrow B$ are $\mathcal{W}$-close, and hence they are homotopic: Indeed, for any $u\in U$ we observe that $d(s'(u),s(a_u))<\bar {\varepsilon }(s(a_u))\leq \delta (s(a_u))$, which implies $s'(u)\in
B(s(a_u),\delta (s(a_u)))$. Therefore, we have:
$$p(s'(u))\in p(B(s(a_u),\delta (s(a_u))))\subset B(a_u,\varepsilon (a_u))$$
\noindent indicating that $p(s'(u))\in B(a_u,\varepsilon (a_u))$.
Furthermore, we note that $d(u,a_u)<\bar{\varepsilon }(s(a_u))\leq \varepsilon (p(s(a_u)))=\varepsilon (a_u),$ which means that $u\in B(a_u,\varepsilon
(a_u)).$

Using this argument for covers we conclude the proof.
\end{proof}

\begin{remark}
We observe that the previous result is equally valid for any continuous map $f:X\rightarrow Y$ between ANR spaces, not necessarily a fibration. In fact, a result from T. Miyata \cite[Th. 2.1]{Miy} ensures that $f$ can be factored as $f=pq$, where $q:X\rightarrow E$ is a homotopy equivalence with $E$ being an ANR space, and $p:E\rightarrow Y$ is a map having a property that is slightly stronger than the usual homotopy lifting property required for a Hurewicz fibration. Specifically, $p$ is a map having the so-called \emph{strong homotopy lifting property} (SHLP) with respect to every space. This means
that for any commutative diagram with solid arrows
$$\xymatrix{
{Z} \ar[r]^h \ar[d]_{i_0} & {E} \ar[d]^p \\
{Z\times I} \ar[r]_H \ar@{.>}[ur]^{\widetilde{H}} & {Y}
}$$ \noindent there exists a map $\widetilde{H}:Z\times I\rightarrow E$ such that $\widetilde{H}i_0=h,$ $p\widetilde{H}=H$ and $\widetilde{H}$ is constant on $\{z\}\times I$ whenever $H$ is constant on $\{z\}\times I$.

Then, by using the homotopy invariance of $\secat (-)$ along with its generalized counterpart (see Proposition \ref{invariance}) and Theorem \ref{chulo} above, we obtain $\secat(f)=\secat _g(f)$. Extending this argument further and once again leveraging the homotopy invariance of $\secat (-)$ and $\secat _g(-)$, we can conclude that Theorem \ref{chulo} above holds true even when $f$ is any continuous map between spaces having the homotopy type of a CW-complex (or an ANR space).
\end{remark}

\begin{remark}
We also want to emphasize that Miyata's article establishes a rather interesting and comprehensive result beyond what we just mentioned in the previous remark. If \textbf{ANR} denotes the full subcategory of the category \textbf{Top} of topological spaces and continuous maps, then it is proved in \cite{Miy} that \textbf{ANR}, along with the class of strong fibrations (i.e., maps having the SHLP with respect to every space) and the class of usual homotopy equivalences, has the structure of a \emph{fibration category} in the sense of Baues \cite{B}. Fibration categories serve as a foundational concept in axiomatic homotopy theory, providing a framework for studying homotopy-related phenomena in a broad and abstract setting.
\end{remark}

\medskip
Theorem \ref{chulo} gives rise to various compelling consequences and outcomes. Among these, we can highlight the following:

\begin{proposition}\cite[Cor. 2.8]{GC}
Consider a space $X$ having the homotopy type of a CW-complex (or equivalently, of an ANR). In this context, we can assert that $\tc(X) = \tc_g(X).$
\end{proposition}

\begin{proof}
Since $\tc_g(-)$ and $\tc(-)$ are homotopy invariants, we can assume that $X$ is an ANR. In this scenario, it is well-known that both $X\times X$ and $X^I$ are ANR spaces, as stated in \cite[Chapt. II, Prop. 4.4]{Hu} and \cite[Chapt. VI, Th. 2.4]{Hu}, respectively. Consequently, by applying Theorem \ref{chulo} to the fibration $\pi: X^I \rightarrow X\times X$, we can successfully complete the proof.
\end{proof}

We can also reaffirm a previously established result, proven by T. Srinivasan in \cite{Sr, Sr2}:

\begin{proposition}\cite[Th. 3.2]{Sr2}, \cite[Cor. 2.10]{GC}
Consider a space $X$ having the homotopy type of a connected CW-complex (or equivalently, of a connected ANR space). In this context, we can state that $\cat(X) = \cat_g(X)$.
\end{proposition}

\begin{proof}
Given that $PX$ is an ANR, one can employ a similar argument as in the previous result for the path fibration $p: PX \rightarrow X.$
\end{proof}

Another interesting example of numerical homotopy invariant agreeing with its generalized version (i.e., using general covers) is given by the sequential topological
complexity $\tc_r(-)$, also called higher topological complexity. This interesting invariant was introduced by Y.B. Rudyak in \cite{R}. The $r$-th sequential topological complexity provides a topological measure for assessing the complexity of the motion planning problem, in which the robot is tasked with sequentially visiting $r$ specified stages. In mathematical terms, if we consider a path-connected space X, the $r$-sequential topological complexity of X is defined as the sectional category of the fibration:
$$\pi _X^r:X^{J_r}\rightarrow X^r,\hspace{10pt}\alpha \mapsto (\alpha (1_1),...,\alpha (1_r))$$
\noindent Here, $J_r$ represents the wedge of $r$ copies of the closed unit interval $I=[0,1]$ with each interval having the base point at $0\in [0,1]$, and $1_i$ indicating the value $1$ in the $i$-th interval. The fibration $\pi _X^r$ is a fibrational substitute of the $r$-diagonal map $\Delta _r:X\rightarrow X^r$ so $\tc _r(X)=\secat (\Delta _r)$ as well. We can straightforwardly define the generalized $r$-sequential topological complexity of a space $X$, denoted as $\tc_{r,g}(X)$, as the generalized sectional category of  $\pi _X^r$.

\begin{proposition}
Consider a space $X$ having the homotopy type of a CW-complex (or equivalently, of an ANR). In this context, we can assert that $\tc _r(X) = \tc_{r,g}(X).$
\end{proposition}

\begin{proof}
Using a similar argument to the one presented in Corollary \ref{hom-inv}, it can be verified that $\tc_{r,g}(-)$ is a homotopy invariant. Consequently, we can assume that $X$ is an ANR space. In this setting, both $X^r$ and $X^{J_r}$ are also ANR spaces, as supported by \cite[Chapt. II, Prop. 4.4]{Hu} and \cite[Chapt. VI, Th. 2.4]{Hu}. Therefore, we can apply Theorem \ref{chulo}.
\end{proof}

As explained in \cite{B-G-R-T}, there are other fibrations (which may not necessarily provide fibrational substitutes for the iterated diagonal) that can be used to define $\tc _r$. In all these examples, the spaces involved are ANR spaces whenever $X$ is one.

\medskip
We also discover an implication related to the topological complexity of a fibration. This can be compared with Theorem \ref{pavesic}.

\begin{proposition}
Let $f:X\rightarrow Y$ be a fibration between ANR spaces. Then $\tc (f)=\secat _g(\pi _f)$.
\end{proposition}

\begin{proof}
Simply note that $\tc(f) = \secat(\pi_f)$, where $\pi_f$ represents a fibration between ANR spaces.
\end{proof}

Theorem \ref{chulo} also has interesting implications for the so-called homotopic distance between two maps, a concept introduced by E. Mac\'{\i}as-Virg\'os and D. Mosquera-Lois in \cite{MV-ML}. Consider two maps, $\varphi, \psi :X\rightarrow Y.$ The homotopic distance $\D(\varphi,\psi)$ between $\varphi$ and $\psi $ is the smallest non-negative integer $n\geq 0$ for which there exists an open covering $X=U_0\cup \cdots \cup U_n$ such that the restrictions $\varphi |_{U_j}$ and $\psi |_{U_j}$ are homotopic maps, for all $j\in \{0,1,\cdots ,n\}.$ If no such covering exists, then we set $\D(\varphi,\psi) =\infty .$

There is a useful characterization in terms of the sectional category. To establish this, we consider the following pullback:
$$\xymatrix{
{\mathcal{P}(\varphi ,\psi)} \ar[d]_{\pi _Y^*} \ar[r] & {Y^I} \ar[d]^{\pi _Y} \\
{X} \ar[r]_{(\varphi ,\psi )} & {Y\times Y} }$$

\begin{lemma}\cite[Th. 2.7]{MV-ML}\label{dist}
The homotopic distance of $\varphi, \psi :X\rightarrow Y$ can be expressed as $\D(\varphi,\psi)=\secat (\pi _Y^*:\mathcal{P}(\varphi ,\psi)\rightarrow X)$.
\end{lemma}

\begin{remark} Considering the generalized homotopic distance $\D _g(\varphi ,\psi)$, where arbitrary subsets are taken instead of open sets, one can easily check that the previous result also applies, expressing $\D _g(\varphi ,\psi)=\secat _g(\pi _Y^*:\mathcal{P}(\varphi ,\psi)\rightarrow X)$.
\end{remark}

In general, taking any two maps $X\stackrel{\varphi }{\longrightarrow }Y\stackrel{\psi }{\longleftarrow }X'$, it is not difficult to check that the double mapping track, or standard homotopy pullback, of $\varphi $ and $\psi$
$$W_{\varphi ,\psi}=\{(x,x',\alpha )\in X\times  X'\times Y^I:\alpha (0)=\varphi (x),\alpha (1)=\psi (x')\}$$ \noindent is an ANR space, provided that $X$, $X'$ and $Y$ are ANR spaces. This will be helpful in the next result:

\begin{proposition}
If $\varphi, \psi :X\rightarrow Y$ are maps between spaces having the homotopy type of a CW-complex (equivalently, of an ANR space), then $\D (\varphi ,\psi)=\D _g(\varphi, \psi )$.
\end{proposition}

\begin{proof}
We can assume, by \cite[Prop. 3.13]{MV-ML} (one can check that it is also valid for the generalized case) that $X$ and $Y$ are ANR spaces. Now, consider the following commutative diagram involving two consecutive pullbacks, where the vertical maps are fibrations:
$$\xymatrix{
{\mathcal{P}(\varphi ,\psi)} \ar[r] \ar[d]_{\pi ^*_Y} & {W_{\varphi ,\psi}} \ar[d]^{\pi '_Y} \ar[r] & {Y^I} \ar[d]^{\pi _Y} \\
{X} \ar[r]_{\Delta _X} & {X\times X} \ar[r]_{\varphi \times \psi } & {Y\times Y} }$$
Given that $X$ is an ANR, we can infer that $X\times X$ is also an ANR. Consequently, according to \cite[Chapt. IV, Th. 3.2]{Hu}, the diagonal map $\Delta _X$ is a cofibration, that is, a closed map having the homotopy extension property. Utilizing \cite[Th. 12]{Str} and \cite[Chapt. IV, Th. 3.2]{Hu} once more, we establish that $\mathcal{P}(\varphi ,\psi)$ is itself an ANR. The desired result follows from combining Theorem \ref{chulo} and Lemma \ref{dist} above along its generalized counterpart.
\end{proof}

\medskip

Finally, there are also consequences with a variation of topological complexity called \emph{symmetric topological complexity}, introduced by M. Farber and M. Grant in \cite{F-G}. In this variation, the requirement for motion planners is that the motion from a state $x$ to itself must be static at $x$, while the motion from $y$ to $x$ reverses the movement taken from $x$ to $y$. This formalization is detailed as follows: Consider $F(X, 2)$ the configuration space representing ordered pairs of distinct points in $X$, and let $\pi_X^0: X^I_0 \rightarrow F(X, 2)$ be the restricted fibration derived from $\pi_X: X^I \rightarrow X \times X$. Consequently, $X^I_0=\pi _X^{-1}(F(X,2))$ is the subset of $X^I$ obtained by excluding the free loops on $X$. The group $\mathbb{Z}_2$ acts freely on both $X^I_0$ and $F(X, 2)$, transforming a path $\gamma$ into its inverse path $\bar{\gamma}$ (i.e., $\bar{\gamma}(t)=\gamma (1-t)$) in the former, and interchanging coordinates in the latter. Moreover, $\pi_X^0$ turns out to be a $\mathbb{Z}_2$-equivariant map. Consider $\widetilde{X}^I_0 := X^I_0 / \mathbb{Z}_2$ and $B(X, 2) := F(X,2) / \mathbb{Z}_2$ the corresponding orbit spaces, with
$\widetilde{\pi}_X^0: \widetilde{X}^I_0 \rightarrow B(X, 2)$ denoting the induced fibration by $\pi_X^0$. Then the symmetric topological complexity of $X$ is defined as $$\tc^S(X) := \secat(\widetilde{\pi }_X^0)+1$$
As noted by M. Farber and M. Grant, this is not a homotopy invariant.

We can obviously define the \emph{generalized symmetric topological complexity} of $X$ as $\tc^S_g(X) := \secat _g(\widetilde{\pi }_X^0)+1$.
To verify the equality $\tc^S(X)=\tc^S_g(X)$ for ANR spaces, we require the following lemma:

\begin{lemma}\label{finito}
If a finite group $G$ acts freely on an ANR space $X$, then the orbit space $X/G$ is an ANR space.
\end{lemma}

\begin{proof}
Note that, as the action is properly discontinuous, the orbit space projection $p:X\rightarrow X/G$ is a covering space projection which leads us to easily check that $X/G$ can be covered by open ANR subspaces. Furthermore, since $p$ is a perfect map we have that $X/G$ is also metrizable (see \cite[Th. 4.4.15]{Eng}). Hence, the result follows from \cite[Chapt. 3, Th. 8.1]{Hu} (see also \cite[Th. 3.3]{Han}).
\end{proof}

\begin{proposition}
If $X$ is an ANR space, then $\tc^S(X)=\tc^S_g(X).$
\end{proposition}

\begin{proof}
As $F(X,2)$ and $X^I_0$ are open subsets of $X\times X$ and $X^I$, respectively, it follows that they are ANR spaces. The result follows from Lemma \ref{finito} above and Theorem \ref{chulo}.
\end{proof}

\section{Generalized monoidal topological complexity}

An intriguing variant of topological complexity is the so-called \emph{monoidal topological complexity}, introduced by N. Iwase and M. Sakai in \cite{I-S}. This arises when, in the algorithm or local rule of the motion planner it is required that the motion be static when the initial and final states of the mechanical system coincide, which is a reasonable requirement. Moving to a more technical context, consider a space $X$ and a subset $A \subset X$ that includes the diagonal, i.e., $\Delta_X(X) \subset A$. A local section $s: A \rightarrow X^I$ of $\pi_X$ is termed \emph{reserved} if, for all $x \in X$, $s_i(x, x) =c_x$ is the constant path at $x.$

\begin{definition}\cite{I-S}\label{tcM}
The \emph{monoidal topological complexity} of a space $X$, denoted as $\tc^M(X),$ is defined as the smallest nonnegative integer $n$ such that $X\times X$ can be covered by $n+1$ open subsets $X\times X=U_0\cup U_1\cup \cdots \cup U_n,$ each of which has a reserved local section of $\pi_X.$ If such an integer $n$ does not exist, then we set $\tc^M(X)=\infty.$
\end{definition}

If the space $X$ is an ENR (Euclidean Neighborhood Retract), A. Dranishnikov noted in \cite{Dr} that the definition of $\tc^M(X)$ can be relaxed in such a way that $\Delta_X(X)$ does not necessarily have to be contained in every open domain $U_i$ of $X\times X$ and $s_i(x,x)=c_x$ whenever $(x,x)\in U_i.$ This observation was further demonstrated by J. Aguilar-Guzm\'an and J. Gonz\'alez in the more general case where $X$ is an ANR. Using their notation, let $\tc^{DM}(X)$ represent this relaxed notion of monoidal topological complexity applied to a general space $X$ (letter $D$ stands for Dranishnikov-type version of monoidal topological complexity).

\begin{proposition}\cite[Prop. 2.3]{A-G}\label{Dranish-M}
If $X$ is an ANR space, then $\tc^{DM}(X)=\tc^{M}(X)$.
\end{proposition}

We also introduce the concept of generalized monoidal topological complexity, denoted as $\tc^M_g(X),$ by considering arbitrary subsets with reserved local sections. Our objective now is to establish that $\tc^M(X) = \tc^M_g(X)$ holds true when $X$ is an ANR. For that, we begin by observing some interesting characteristics of monoidal topological complexity.

Undoubtedly, monoidal topological complexity serves as an upper bound for topological complexity. Moreover, for locally finite simplicial complexes, N. Iwase and M. Sakay proved that $\tc (X)\leq \tc ^M(X)\leq \tc (X)+1$ (see \cite[Th. 1 (erratum)]{I-S}).
However, it is crucial to note that this is a distinct invariant from topological complexity. In contrast to topological complexity, $\tc ^M(-)$ does not exhibit homotopy invariance. To illustrate this, consider the subspace $X$ of $\mathbb{R}^2$ defined as
$X=\bigcup _{n=0}^{\infty }L_n.$ Here, each $L_n$ is a line segment joining $(0,0)$ to $(1,1/n)$ for $n\geq 1$, and $L_0$ is the line segment $[0,1]\times {0}$.
It is evident that $X$ is a contractible space. However, we find that $\tc ^M(X)\neq 0$. To see why, suppose $\tc ^M(X)=0$, and let $s:X\times X\rightarrow X^I$ be a globally continuous section of the path fibration $\pi:X^I\rightarrow X\times X$ such that $s(x,x)=c_x$. Now, for $n\geq 1$, consider the sequence of points $x_n=(1,1/n)$ and $x_0=(1,0)$. It is evident that as $n\rightarrow \infty$, $(x_n,x_0)\rightarrow (x_0,x_0)$ in $X\times X$, and, by continuity, we should have $s(x_n,x_0)\rightarrow s(x_0,x_0)=c_{x_0}$. However, this is not possible, as for all $n\geq 1$, $s(x_n,x_0)$ represents a path from $x_n$ to $x_0$ that must pass through the origin $(0,0).$

This example also illustrates that, in general, $\tc (X)\neq \tc ^M(X)$. It is important to note that this observation does not contradict Iwase-Sakai's conjecture as stated in \cite{I-S}:

\medskip
\textbf{Iwase-Sakai Conjecture.}
For any locally finite simplicial complex $X$, $\tc(X)=\tc ^M(X)$.
\medskip

\begin{figure}
\begin{center}
\includegraphics[scale=0.4]{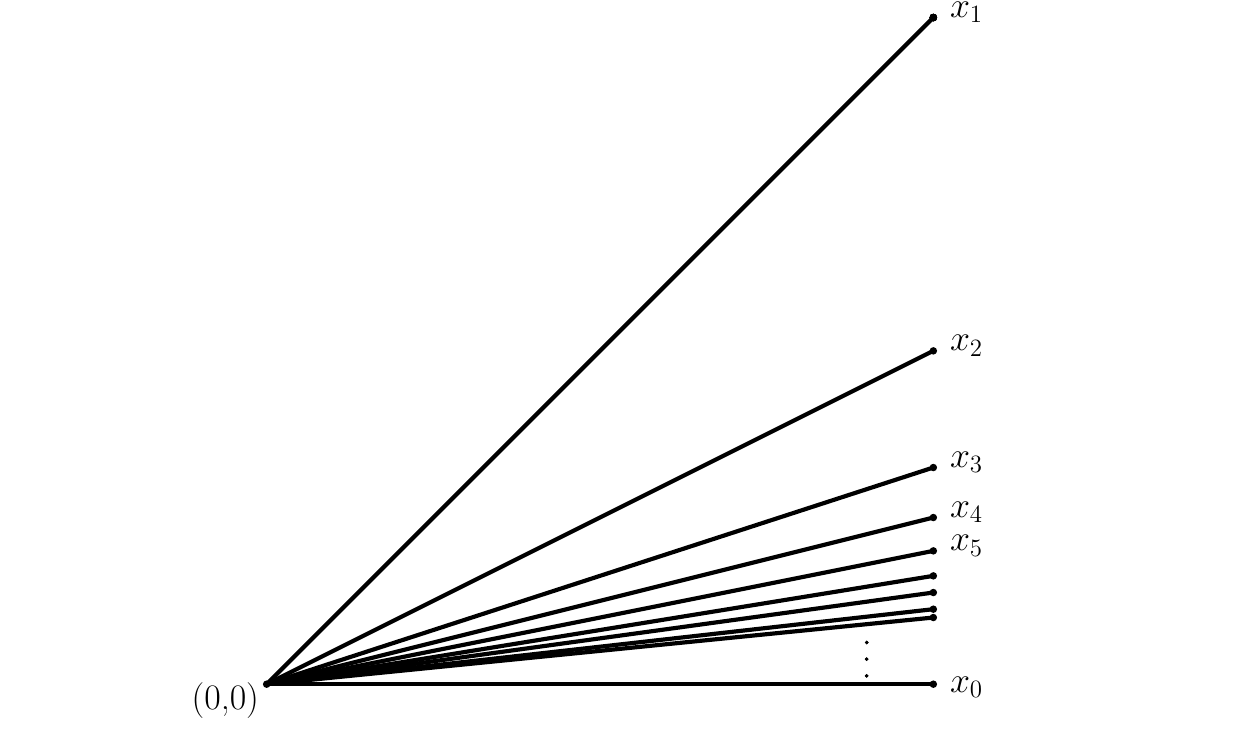} \\
\end{center}
\caption{Example illustrating that $\tc^M(-)$ is not a homotopy invariant}
\end{figure}

\begin{remark}\label{rem-iwase-sakai} The Iwase-Sakai conjecture is still unresolved and remains an open problem. However, there are sufficient conditions in a space for the Iwase-Sakai conjecture to be positively fulfilled. In this regard, A. Dranishnikov proved in \cite[Th. 2.5]{Dr} that the equality $\tc (X)=\tc ^M(X)$ holds provided that $X$ is a $k$-connected simplicial complex (in actual fact, also valid for CW-complexes) satisfying $(k+1)(\tc (X)+1)>{\sf {dim}}(X)+1$. Other partial results exist where the Iwase-Sakai conjecture is positively verified. We can mention, for example, a result proved by Dranishnikov \cite[Lem. 2.7]{Dr}, establishing that for a connected Lie group $G$, we have $\tc(G)=\cat(G)=\tc^M(G)$.
Here, the equality $\tc(G)=\cat(G)$ had been previously proved by M. Farber in \cite[Lem. 8.2]{F3}. Another compelling example supporting the Iwase-Sakai conjecture, which extendes the previous result by Dranishnikov, is presented by J. Aguilar-Guzm\'an and J. Gonz\'alez in \cite[Th. 1.1]{A-G}. This example concerns polyhedral products defined by LS-logarithmic families of locally compact connected CW topological groups.
Explicitly, let $\underline{G}^K$ denote the polyhedral product associated with an abstract simplicial complex $K$ having a vertex set $\{1,\cdots ,m\}$, along with a based family $\underline{G}=\{(G_i,e_i)\}_{i=1}^m$ of locally compact connected CW topological groups, where $e_i$ represents the neutral element of $G_i$. If $\underline{G}$ forms an LS-logarithmic family, then we have
$$\tc (\underline{G}^K)=\tc ^M(\underline{G}^K)=\max \left\lbrace \sum _{i\in \sigma _1\cup \sigma _2}\cat (G_i):\sigma _1,\sigma _2\in K\right\rbrace.$$
The proof methodology of this result essentially employs a Fadell-Husseini perspective on monoidal topological complexity. We will elucidate this approach at the end of this section (refer to Proposition \ref{final}).
\end{remark}

As just mentioned, $\tc ^M(-)$ is not a homotopy invariant. However, when we restrict ourselves to a broad class of spaces known as locally equiconnected spaces, homotopy invariance is achieved through the so-called relative category in the sense of Doeraene-El Haouari \cite{D-H}. At this juncture, it is necessary for us to pause and offer a more detailed explanation of this matter. Firstly, if $i_X:A\hookrightarrow X$ is a cofibration, we can consider, for each nonnegative integer $n$, a subspace $T^n(i_X)$ of $X^{n+1}$ as follows:
$$T^n(i_X)=\{(x_0,x_1,...,x_n)\in X^{n+1}\hspace{3pt}:x_i\in A\hspace{3pt}\mbox{for some}\hspace{3pt}i\}.$$
\noindent The corresponding inclusion is denoted as $t_n:T^n(i_X)\hookrightarrow X^{n+1}$, and it is, in fact, a cofibration. This inclusion $t_n$ is referred to as the $n$-th sectional fat-wedge of $i_X$.

\begin{definition}
The \emph{relative category} of a cofibration $i_X:A\hookrightarrow X$, denoted as ${\sf {relcat}}(i_X)$, is defined as the least nonnegative integer $n$ (or infinity) such that there exists a map $f:X\rightarrow T^n(i_X)$ making the following diagram homotopy commutative :
$$\xymatrix{
{A} \ar[rr]^{\tau _n} \ar@{^{(}->}[d]_{i_X} & & {T^n(i_X)}
\ar@{^{(}->}[d]^{t_n} \\
{X} \ar@{.>}[urr]^f \ar[rr]_{\Delta _{n+1}} & & {X^{n+1}}   }$$
Here $\tau _n$ denotes the diagonal map $\Delta _n:X\rightarrow X^{n+1}$ restricted to $A$.
\end{definition}

\begin{remark}
We must caution the reader that, in fact, Doeraene and El Haouari introduced the relative category for any map, not necessarily a cofibration, using the join construction --a tool more in line with category theory-- which is a combination of homotopy pushouts and pullbacks (see \cite[Def. 2]{D-H} for the original definition, \cite[Prop. 26]{D-H} for its Whitehead characterization and \cite[Cor. 11]{C-V}, \cite[Prop. 19]{C-G-V} for the case of a cofibration). We have chosen not to consider this more general definition because we believe it goes beyond the intended scope of this survey.
\end{remark}

Now, recall that a space $X$ is said to be locally equiconnected (LEC, for short) if the diagonal map $\Delta _X:X\rightarrow X\times X$ is a cofibration (meaning that it is a closed map satisfying the homotopy extension property). The class of LEC spaces is notably extensive, encompassing well-established examples such as CW-complexes and ANR spaces.
Since the relative category is a homotopy invariant, it follows, as a consequence of the following result by J.G. Carrasquel-Vera, J.M. Garc\'{\i}a-Calcines and L. Vandembroucq that $\tc ^M(-)$ is a homotopy invariant when restricted to LEC spaces (also, compare with \cite[Rem. 1.4]{I-S}).

\begin{proposition}\cite[Th. 12]{C-G-V}\label{bridge}
If $X$ is a LEC space, then $$\tc ^M(X)={\sf {relcat}}(\Delta _X:X\rightarrow X\times X).$$
\end{proposition}

After examining the relative category of a cofibration, we encounter a characterization involving open coverings, which aligns more closely with our philosophical approach. Specifically, for a cofibration $i_X: A \hookrightarrow X$, we define a subset $U \subset X$ as \emph{relatively sectional} if $A \subset U$ and there exists a homotopy of pairs
$H:(U \times I, A \times I) \rightarrow (X, A)$ such that $H(x, 0) = x$ and $H(x, 1) \in A$ for all $x \in U$.

\begin{proposition}\cite[Th. 1.6]{GC}\label{car-relcat}
If $i_X:A\hookrightarrow X$ is a cofibration where $X$ is a normal
space, then ${\sf {relcat}}(i_X)$ the least nonnegative integer $n$ (or infinity) such that $X$ admits a cover constituted by $n+1$ relatively
sectional open subsets.
\end{proposition}

We can directly define the generalized relative category of a cofibration, as follows:

\begin{definition}
Let $i_X:A\hookrightarrow X$ be a cofibration. The
\emph{generalized relative category} of $i_X,$
${\sf {relcat}}_g(i_X),$ is defined as the least nonnegative
integer $n$ (or infinity) such that $X$ admits a cover by $n+1$
relatively sectional subsets.
\end{definition}

Similar to the approach taken for the generalized sectional category, the generalized relative category is a homotopy invariant, as shown in \cite[Prop. 2.12]{GC}. Moreover, through a comprehensive examination utilizing tools primarily derived from Subsection \ref{calcines}, it has been established that when addressing ANR spaces, the equality ${\sf {relcat}}(i_X) = {\sf {relcat}}_g(i_X)$ holds:

\begin{theorem}\cite[Th. 2.16]{GC}\label{chulo2}
Let $i_X:A\hookrightarrow X$ be a cofibration between ANR spaces.
Then ${\sf {relcat}}(i_X)={\sf {relcat}}_g(i_X).$
\end{theorem}

Upon consideration of Proposition \ref{bridge} along with Theorem \ref{chulo2} mentioned earlier, we deduce that if $X$ is an ANR, then $\tc ^M(X)={\sf {relcat}}_g(\Delta _X)$. Finally, it becomes clear that the monoidal topological complexity aligns with its generalized version when the examined space is an ANR.

\begin{proposition}\label{cor-chulo2}
If $X$ is an ANR space, then $\tc ^M(X)=\tc ^M_g(X)$.
\end{proposition}

\begin{proof}
Since any subset $A\subset X\times X$ with a reserved local section $s$ gives rise to a homotopy $H(x,y,t):=(s(x,y)(t),y)$ making $A$ a relatively sectional subset with respect to the cofibration $\Delta _X,$ we obtain the inequality ${\sf {relcat}}_g(\Delta _X)\leq \tc ^M_g(X)$. Therefore
$$\tc ^M(X)={\sf {relcat}}_g(\Delta
_X)\leq \tc ^M_g(X)$$ \noindent concluding that $\tc ^M(X)=\tc ^M_g(X).$
\end{proof}

To conclude this work, we will explain the Fadell-Husseini-like approach to monoidal topological complexity announced at the end of Remark \ref{rem-iwase-sakai}, as presented by J. Aguilar-Guzm\'an and J. Gonz\'alez \cite{A-G}. This approach is inspired by the definition of relative category in the sense of Fadell-Husseini, introduced in \cite{F-H}.

\begin{definition}\label{FHM}
The \emph{Fadell-Husseini monoidal topological complexity} of a space $X$, denoted by $\tc ^{FH}(X)$, is the smallest nonnegative integer $n$ such that $X\times X$ can be covered by $n+1$ open subsets, $X\times X=U_0\cup U_1\cup \cdots \cup U_n$, on each of which there exists a continuous section $s_i:U_i\rightarrow X^I$ of $\pi _X$ such that:
\begin{enumerate}
\item[(1)] $U_0$ contains the diagonal $\Delta _X(X)$;

\item[(2)] $s_0(x,x)=c_x$, the constant path at $x,$ for all $x\in X$;

\item[(3)] $\Delta _X(X)\cap U_i=\emptyset$ for all $i\geq 1.$
\end{enumerate}
If such an integer $n$ does not exist, then $\tc ^{FH}(X)=\infty $. The Fadell-Husseini \emph{generalized} monoidal topological complexity, denoted by $\tc _g^{FH}(X)$, is similarly defined using coverings that are not necessarily open.
\end{definition}

\begin{remark}\label{remarcon}
If $X$ is a Hausdorff space (in other words, the diagonal $\Delta _X(X)$ is closed in $X\times X$), it follows directly from the definitions that  $\tc^{DM}(X)\leq \tc ^{FH}(X)\leq \tc ^M(X).$ Consequently, Proposition \ref{Dranish-M} and Proposition \ref{cor-chulo2} prove that
$$\tc^{DM}(X)=\tc ^{FH}(X)=\tc ^M(X)=\tc ^M_g(X)$$ \noindent whenever $X$ is an ANR space.
\end{remark}

As pointed out by J. Aguilar-Guzm\'an and J. Gonz\'alez, condition (3) in Definition \ref{FHM} may be disregarded without changing the value of $\tc _g^{FH}(X)$, since the diagonal can be excluded, if necessary, from the sets $U_1,...,U_n.$ A comparable observation holds true in the non-generalized context, provided that $X$ is a Hausdorff space. Furthermore, they also present the following remarkable result:

\begin{proposition}\cite[Th. 1.4]{A-G}\label{overlook}
If $X$ is a Hausdorff LEC space with $X\times X$ being normal, then we can safely overlook condition (2) in Definition \ref{FHM} without affecting the resulting value of $\tc^{FH}(X)$. Likewise, in the generalized setting, if $X$ is an ANR space, the conclusion remains unchanged.
\end{proposition}

Our objective now is to establish that $\tc^{FH}(X)=\tc_g^{FH}(X)$ for an ANR space $X$.  For this purpose, J. Aguilar-Guzm\'an and J. Gonz\'alez introduced Fadell-Husseini type variations for the concepts of ${\sf {relcat}}(i_X)$ and ${\sf {relcat}}_g(i_X)$, where $i_X:A\hookrightarrow X$ represents a cofibration.

\begin{definition}
Let $i_X:A\hookrightarrow X$ be a cofibration. The
\emph{Fadell-Husseini relative category} of $i_X,$ denoted by
${\sf {relcat}}^{FH}(i_X),$ is defined as the least nonnegative
integer $n$ (or infinity) such that $X$ admits an open cover $\{U_i\}_{i=0}^n$ satisfying:
\begin{enumerate}
\item[(a)] $U_0$ is relatively sectional;

\item[(b)] for $i\geq 1,$ $U_i\cap A=\emptyset $ and there are homotopies $H_i:U_i\times I\rightarrow X$ such that $H_i(x,0)=x$ and $H_i(x,1)\in A,$ for all $x\in U_i$.

\end{enumerate}
If such an integer does not exist then we set ${\sf {relcat}}^{FH}(i_X)=\infty .$ The Fadell-Husseini \emph{generalized} relative category of $i_X$, denoted by ${\sf {relcat}}_g^{FH}(i_X)$, is similarly defined using general coverings that are not necessarily open.
\end{definition}

Once more, employing methodologies akin to those elucidated in \cite{GC}, we derive the ensuing result:

\begin{theorem}\cite[Prop. 3.3 and Prop. 3.7]{A-G}\label{teorAG}
If $i_X: A \hookrightarrow X$ is a cofibration and $X$ is a normal space, then ${\sf {relcat}}(i_X)={\sf {relcat}}^{FH}(i_X)$. If, in addition, both $A$ and $X$ are ANR spaces, then ${\sf {relcat}}^{FH}(i_X)={\sf {relcat}}(i_X)={\sf {relcat}}_g^{FH}(i_X)$.
\end{theorem}

Finally, we can conclude with the desired result.

\begin{proposition}\cite[Th. 1.3]{A-G}\label{final}
If $X$ is an ANR space, then $\tc^{FH}(X)=\tc_g^{FH}(X)$. Therefore, in this case we have the following chain of equalities
$$\tc ^{FH}(X)=\tc_g^{FH}(X)=\tc^{DM}(X)=\tc ^M(X)=\tc ^M_g(X).$$
\end{proposition}

\begin{proof}
Combining Remark \ref{remarcon}, Proposition \ref{bridge}, and Theorem \ref{teorAG}, we obtain that $\tc^{FH}(X)=\tc^{M}(X)={\sf {relcat}}(\Delta_X)={\sf {relcat}}_g^{FH}(\Delta_X).$
However, ${\sf {relcat}}_g^{FH}(\Delta_X)\leq \tc^{FH}_g(X)$ (and therefore $\tc^{FH}(X)=\tc_g^{FH}(X)$). Indeed, suppose that $\tc_g^{FH}(X)=n$, and let $\{U_i\}_{i=0}^n$ be a cover (not necessarily open) of $X\times X$ satisfying the conditions given in Definition \ref{FHM}. Then, it is immediate to check that for all $i\geq 0$, the expression $H_i(x,y,t):=(s_i(x,y)(t),y)$ gives rise to the inequality ${\sf {relcat}}_g^{FH}(\Delta_X)\leq n.$
\end{proof}

\begin{remark}
As a consequence of Proposition \ref{overlook} and Proposition \ref{final}, we conclude that the Iwase-Sakai conjecture holds true for any ANR space $X$ provided there exists a (not necessarily monoidal) motion planner with $\tc(X)+1$ (not necessarily open) motion planners, one of which contains the diagonal $\Delta_X(X)$ (compare with \cite[Cor. 3 in Erratum]{I-S}).
\end{remark}




\bigskip
\textbf{Funding.}
This work was supported by the project PID2020-118753GB-I00 from the Spanish Ministry of Science and Innovation.

\end{document}